\documentclass[11pt]{article}
\usepackage{amsfonts,amsmath,latexsym,verbatim,amscd,amsthm}

\title{Constant curvature foliations \\ in asymptotically hyperbolic spaces}

\author{Rafe Mazzeo \thanks{Email: mazzeo@math.stanford.edu.
Supported by the NSF under Grant DMS-0505709}\\ Stanford University \and Frank
Pacard \thanks{Email: pacard@univ-paris12.fr, Membre de l'Institut Universitaire de France}\\
Universit\'e Paris 12}

\date{}

\newtheorem{theorem}{Theorem}[section]
\newtheorem{proposition}{Proposition}[section]
\newtheorem{corollary}{Corollary}[section]
\newtheorem{lemma}{Lemma}[section]
\newtheorem{definition}{Definition}[section]

\newcommand{\RR}{\mathbb{R}}

\newcommand{\HH}{\mathbb{H}}
\newcommand{\e}{\varepsilon}
\newcommand{\del}{\partial}

\newcommand{\tr}{\mathrm{tr}\,}

\newcommand{\calC}{{\mathcal C}}
\newcommand{\calF}{{\mathcal F}}

\newcommand{\calL}{{\mathcal L}}

\newcommand{\calN}{{\mathcal N}}
\newcommand{\calO}{{\mathcal O}}
\newcommand{\calR}{{\mathcal R}}

\newcommand{\calU}{{\mathcal U}}
\newcommand{\calV}{{\mathcal V}}
\newcommand{\frakc}{{\mathfrak c}}

\newcommand{\olg}{\overline{g}}
\newcommand{\Ric}{\mathrm{Ric}}
\newcommand{\olN}{\overline{N}}
\newcommand{\II}{I\!I}
\newcommand{\wh}{\widehat}
\newcommand{\calE}{\mathcal{E}}
\newcommand{\Hess}{\mbox{Hess}\,}
\newcommand{\olM}{\overline{M}}

\begin{document}

\maketitle

\begin{abstract}
Let $(M,g)$ be an asymptotically hyperbolic manifold with a smooth conformal compactification.
We establish a general correspondence between semilinear elliptic equations of scalar curvature type
on $\del M$ and Weingarten foliations in some neighbourhood of infinity in $M$. We focus mostly
on foliations where each leaf has constant mean curvature, though our results apply equally well to
foliations where the leaves have constant $\sigma_k$-curvature. In particular, we prove the existence
of a unique foliation near infinity in any quasi-Fuchsian $3$-manifold by surfaces with constant
Gauss curvature. There is a subtle interplay between the precise terms in the expansion for 
$g$ and various properties of the foliation. 
Unlike other recent works in this area, by Rigger \cite{Ri} and Neves-Tian \cite{NT1},
\cite{NT2}, we work in the context of conformally compact spaces, which are 
more general than perturbations of the AdS-Schwarzschild space, but we do assume
a nondegeneracy condition.  
\end{abstract}

\section{Introduction}
A foliation is called geometric if each leaf inherits some particular geometric structure from the  
ambient metric. We are interested here in foliations where the leaves are of codimension one and  
satisfy some Weingarten condition, i.e.\ the principal curvatures $\kappa_1, \ldots, \kappa_n$ 
satisfy $f(\kappa_1, \ldots, \kappa_n) = c$  where $f$ is symmetric in its entries, and the constant
$c$ can vary from  leaf to leaf. The most commonly  studied of these are foliations by minimal hypersurfaces,
$\sum \kappa_j = 0$,  or where the leaves have  constant mean curvature (CMC), $\sum \kappa_j = c$, and 
this latter class will  be our main focus. However,  more general cases are also of interest, e.g.\ 
when $f = \sigma_k$,  the $k^{\mathrm{th}}$ symmetric function of the principle curvatures (in particular, 
when $k=n$,  so $f$ is the Gauss-Kronecker curvature), and consideration of these requires little extra 
effort to incorporate into our main results. 

The main questions we consider here concern the existence and uniqueness of such foliations in some  
neighborhood of infinity in general asymptotically hyperbolic manifolds. For simplicity, we  
concentrate on CMC foliations in most of this paper, and relegate discussion of the minor changes 
needed to handle more general functions $f$ in a final section. There are several motivations for 
studying geometric foliations. On the most basic level, one might  hope to prove that such foliations exist and 
are fairly stable or rigid, and hence are interesting objects more or less uniquely associated to a  
Riemannian manifold. Foliations in an ambient Lorentzian  space with (spacelike) CMC leaves are used 
frequently in relativity, as one part of a `good coordinate gauge' \cite{An}, \cite{Ger}. In the 
Riemannian setting, an influential paper by Huisken and Yau \cite{HY} proved the existence of a 
foliation near infinity in an asymptotically Euclidean manifold  using a geometric heat flow. In 
certain situations this is unique, and they use it to define a  `center of mass' for an isolated 
gravitational system. Essentially the same result was also attained by Ye \cite{Y2} using elliptic 
singular perturbation methods. The sharpest uniqueness statement for foliations of this type was 
obtained by Qing and Tian \cite{QT}. There are analogous results in the asymptotically hyperbolic 
setting. Existence of CMC foliations on high order perturbations of the  AdS Schwarzschild space was
proved by Rigger \cite{Ri}, again using mean curvature flow, and quite recently Neves and Tian
\cite{NT1}, \cite{NT2} have established uniqueness and extended the existence theory in this setting. In a 
somewhat different direction, some time ago, Labourie \cite{La} used pseudoholomorphic curves in 
the cotangent  bundle to construct constant Gauss curvature foliations near infinity in convex cocompact
hyperbolic  three-manifolds. Our results are closely related to the results of Rigger, Neves-Tian and Labourie. 

All of these are foliations in a neighbourhood of infinity, but one may also consider foliations in a 
compact set of the manifold which collapse in the limit to some lower dimensional set. Ye \cite{Y1} 
proved, under certain conditions, existence and uniqueness of CMC spheres collapsing to a point. 
It turns out that this limiting point is necessarily a critical point of the scalar curvature function. 
A recent extension of this \cite{PX} treats the `very degenerate' case where the ambient manifold has 
constant scalar curvature. The papers \cite{MP1}, \cite{MMP} construct `partial' CMC foliations which 
collapse to higher dimensional minimal submanifolds. This minimality is again necessary. The survey 
\cite{Pac-surv} gives a good overview of all of this. 

As already noted, our goal here is to revisit this problem in the asymptotically hyperbolic case. We shall 
work in a broader geometric setting than either Rigger or Neves-Tian, namely that of conformally compact  
manifolds $(M,g)$. Thus $M$ is a compact $(n+1)$-dimensional manifold with boundary, with $n \geq 2$,  
and $g = \rho^{-2}\, \olg$ is a complete metric  on its interior; here $\olg$ is a metric which extends 
up to $\del M$ and $\rho$ is a smooth defining  function for $\del M$. In this paper, unless otherwise stated, 
we require that $\bar g$ has a $\calC^{3,\alpha}$ extension up to $\del M$. Assuming $|d\log \rho|_g^2 \to 1$
as $\rho \to 0$, then $g$ is asymptotically hyperbolic in the sense that the sectional curvatures all 
tend to $-1$ at infinity.  Naturally associated to $g$ is its conformal infinity, 
\begin{equation}
\frakc(g) : = \big[\left. \olg \right|_{T\del M} \big],
\label{eq:confinf}
\end{equation} 
which is a conformal class on $\del M$. There is a simple correspondence, due to Graham and Lee \cite{GL}, 
between metrics on $\del M$ which represent this conformal class, `special' boundary defining functions, 
and hypersurfaces near infinity in $M$, which are essential and outer convex, which are the level sets of 
these defining functions. More specifically, given $h_0 \in \frakc(g)$, there is a boundary
defining function $x$ so that 
\begin{equation}
g = \frac{dx^2 + h(x)}{x^2}, \qquad \mbox{where}\ h(x) =  h_0 + h_1 \, x +  h_2 \, x^2 + \ldots .
\label{eq:norform}
\end{equation}

The level sets $\{x = \mbox{const.}\}$ have mean curvature which is almost constant, and in our main 
existence results we show how to perturb these level sets so that they are exactly CMC (or have
constant $\sigma_k$ curvatures, etc.). The key to our 
method, however, is to relate this problem about the extrinsic geometry of these level sets to conformal 
geometry problems in the class $\frakc(g)$. 

In dimension $n \geq 3$, the first, and most important case, is when $h_1 = 0$ and $h_2$ is equal to the 
negative of the Schouten tensor 
of $h_0$: 
\[
h_2 = - P_{h_0}:  = - \frac{1}{n-2}\left(\Ric(h_0) - \frac{R_{h_0}}{2(n-1)} \, h_0\right). 
\] 
As explained in \S 2, this corresponds to the initial part of the expansion of a Poincar\'e-Einstein metric. 
It will emerge why these conditions are well-defined. 
\begin{theorem}
Let $(M^{n+1}, g)$ be conformally compact, and suppose that for some smooth boundary defining function $x$, the
conformal compactification $\olg = x^2 g$ is $\calC^{3,\alpha}$ up to $\partial M$. Suppose also that 
$h_1 = 0$ and $h_2 = -P_{h_0}$. 
\begin{enumerate}
\item[(i)] If the conformal class $\frakc(g)$ has negative Yamabe invariant, then there exists a 
unique CMC foliation near infinity.
\item[(ii)] If, on the other hand, $\frakc(g)$ has positive Yamabe invariant, then to each constant 
scalar  curvature metric $h_0 \in \frakc(g)$ which is nondegenerate for the linearized Yamabe equation,  
we can associate a CMC foliation. Different constant scalar curvature metrics correspond to geometrically distinct foliations.
\end{enumerate}
\label{th:pa1}
\end{theorem}
Recall that a conformal class $\frakc(g)$ is said to have negative or positive Yamabe invariant if, for any
representative $h_0 \in \frakc(g)$, the least eigenvalue $\lambda_1$ of the conformal Laplacian
\[
- \left( \Delta_{h_{0}} - \frac{n-2}{4(n-1)} R_{h_{0}} \right)
\]
is negative, respectively positive. Here $R_{h_{0}}$ is the scalar curvature of the metric $h_{0}$ on $\partial M$. 
When $\frakc(g)$ has negative Yamabe invariant, there is a unique constant scalar curvature metric in this conformal
class (up to scale), while if this Yamabe is positive, there may very well be a large number of constant scalar curvature 
representatives, and hence a large number of geometrically distinct CMC foliations near infinity. 

Theorem \ref{th:pa1} follows from a more general result. In terms of the expansion (\ref{eq:norform}), define
the two functions 
\begin{equation}
\kappa_1  : =  \frac{1}{2}\, \tr^{h_0} h_1 \in \calC^{2,\alpha}(\del M) \qquad \mbox{and} \qquad \kappa_2  : = 
tr^{h_0} h_2 - \frac{1}{2} \, \| h_1\|^2_{h_0} \in \calC^{1,\alpha}(\del M).
\label{eq:k1k2}
\end{equation}
We show later that these are independent of the choice of representative $h_0 \in \frakc(g)$. Assuming the role of 
the conformal Laplacian is the operator
\begin{equation}
\mathbb L_{h_0} : = - \left( \Delta_{h_{0}} + \frac{n-2}{2} \kappa_{2} \right),
\label{eq:genconflap}
\end{equation}
which we call the {\em generalized conformal Laplacian}. Note that if $h_1 = 0$ and $h_2 = -P_{h_0}$, then
${\mathbb L}_{h_0}$ {\it is} the conformal Laplacian.

\begin{definition}
When $n \geq 3$, we say that the conformally compact metric $g$ has positive or negative generalized boundary 
Yamabe invariant if
\[
\inf\ \{\int_{\del M} \phi_0 \, \mathbb L_{h_0}\phi_0\, dV_{h_0}:  \|\phi_0\|_{L^{\frac{2n}{n-2}}=1}\}
\]
is positive or negative (or equivalently,  if the least eigenvalue of $\mathbb L_{h_0}$ is positive or negative). 
When $n =2$, we say that the conformally compact metric $g$ has  positive or negative generalized boundary Yamabe invariant if 
\[
- \int_{\partial M} \kappa_{2} \, dvol_{h_{0}} \, ,
\]
is positive or negative, respectively.
\end{definition}
The somewhat confusing sign conventions here are analogous to the ones in the Yamabe problem. We will see that the
signs of these boundary invariants are  independent of the representative $h_0 \in \frakc(g)$. 

Now consider the Yamabe-type equation 
\begin{equation}
e^{-2 \phi_0} \, \left( \kappa_2 + \Delta_{h_0} \phi_0 + \frac{n-2}{2} \, |\nabla^{h_0} \phi_0|^2_{h_0}  \right) - 
\tilde \kappa_2 = 0.
\label{eq:yyttee}
\end{equation}
where $\tilde \kappa_2$ is constant. Theorem~\ref{th:pa1} follows from 
\begin{theorem}
Let $(M^{n+1},g)$ be conformally compact, $n \geq 2$  and suppose that $\kappa_1 \equiv 0$.  
\begin{enumerate}
\item[(i)] If $g$ has negative generalized boundary Yamabe invariant, then there exists a unique CMC 
foliation near infinity.
\item[(ii)]  If $g$ has positive generalized boundary Yamabe invariant, then to each solution $\phi_0\in 
\calC^{2,\alpha}(\del M)$ of (\ref{eq:yyttee}) for which the linearization of (\ref{eq:yyttee}) at $\phi_0$ is 
invertible, we can associate a CMC foliation.  Different solutions correspond to geometrically distinct foliations.
\end{enumerate}
\label{th:theorem1.1}
\end{theorem}
When $g$ has negative generalized boundary Yamabe invariant, we shall prove that there exists a solution of 
(\ref{eq:yyttee}) with $\tilde \kappa_{2}$ a positive constant, and this solution is unique once 
this constant is fixed. When $g$ has positive generalized boundary Yamabe invariant, it may not be possible 
to find a solution of (\ref{eq:yyttee}), and uniqueness might not hold. 

\medskip

The condition $\kappa_1 \equiv 0$ does not depend on the choice of the representative $h_0 \in \frakc(g)$. Note that
if in addition $\tilde{\kappa}_2 =0$, then invertibility of the linearization of (\ref{eq:yyttee}) at $\phi_0$ 
necessarily fails and our method does not apply in this case, even though CMC foliations may well exist in such 
circumstances (as, e.g. in \cite{Ri} , \cite{NT1} and \cite{NT2}).

When $\kappa_1 \neq 0$ the situation is somewhat more complicated. 
\begin{theorem}
Let $(M^{n+1},g)$ be conformally compact, $n \geq 2$. 
\begin{enumerate}
\item[(i)] Assume that $\kappa_1 > 0$ everywhere and also that there exists a conformal compactification $\bar g$ 
which is $\mathcal C^{3}$ up to $\del M$, then there exists a unique CMC foliation near infinity. 
\item[(ii)] If $\kappa_1 < 0$ everywhere and some conformal compactification $\bar g$ has a $\calC^\infty$ extension 
up to $\partial M$, then there exists a CMC foliation with gaps, with leaves tending to infinity.
\end{enumerate}
\end{theorem}
As before, the condition that $\kappa_1$ does not change sign is independent of the choice of the representative
$h_0 \in \frakc(g)$. If $\kappa_1$ changes sign or vanishes somewhere, our methods do not apply  and the existence 
of CMC foliations is doubtful.  The precise meaning of foliations with gaps will be clear in the course of
the proof of this theorem.

The existence of these foliations is established by a perturbation argument which is particularly straightforward 
when $\kappa_1 = 0$. The other cases are more singular. The  uniqueness statements follow from a certain
monotonicity of the mean curvature function defined on the foliations.  

Under various conditions, we also prove the existence of foliations with leaves having constant $\sigma_k$ curvatures.
Rather than state these results completely here, we note only one special case:
\begin{theorem}
Let $M$ be a three-dimensional quasi-Fuchsian hyperbolic manifold (i.e.\ a geometrically finite 
deformation of a warped product $(\Sigma \times \RR, dt^2 + \cosh^2 t\, h)$ where $(\Sigma,h)$ is a compact 
hyperbolic surface). Then the ends of $M$ admit unique foliations by compact surfaces with 
constant Gauss curvature. 
\end{theorem}
Note that we already know that the ends of these manifolds admit unique CMC foliations by Theorem~\ref{th:theorem1.1}. 

Quite recently we learned of recent work by Espinar, Galvez and Mira \cite{EGM} which is related to all of this.
They too study the correspondence between scalar curvature type equations on the boundary and Weingarten 
hypersurfaces in the interior, but they work only in $\HH^{n+1}$ and use many special properties of that 
space; furthermore, their Weingarten conditions are different from the more familiar CMC condition we study here.

This paper is organized as follows. The next section describes the geometry of conformally compact metrics.
The notion of monotone CMC foliations is introduced and related to uniqueness of the foliation in \S 3.
The main calculations of the second fundamental form and mean curvature of the level sets of special
boundary defining functions, and the effect on these quantities of conformal changes on the boundary,
is carried out in \S 4. This leads to the proofs of the various existence theorems in \S 5. The
final \S 6 discusses the alterations needed to prove analogous results for foliations with constant
$\sigma_k$ curvature leaves. 

\section{Asymptotically hyperbolic metrics}
In this section we review the relevant aspects of the geometry of conformally compact metrics. 

\subsubsection*{Conformally compact metrics}
Let $M$ be the interior of a smooth compact manifold with boundary. A metric $g$ on $M$ is called  conformally  
compact if $g = \rho^{-2}\olg$, where $\olg$ is a metric on the closed manifold with boundary which is smooth 
and nondegenerate up to the boundary, and $\rho$ is a smooth defining  function for $\del M$, i.e.\  $\rho = 0$ 
only on $\del M$ and $d\rho \neq 0$ there. We say that $g$ has a $\calC^{3,\alpha}$ conformal compactification
if for some (and hence any) smooth boundary defining function $\rho$, $\olg$ has a $\calC^{3, \alpha}$ extension up 
to $\del M$. (This is not the most general nor the most invariant way of stating this condition.)  
A brief calculation shows that the curvature tensor of $g$  has the form 
\[
R_{ijk\ell} = -|d\rho|_{\olg}^2 \, (g_{ik}g_{j\ell} - g_{i\ell}g_{jk}) + \calO(\rho),
\] 
so we say  that $(M,g)$ is asymptotically hyperbolic (AH) if $|d\rho|_{\olg} = 1$ when $\rho = 0$.  
This is an intrinsic condition, i.e.\ is independent of the choice of factors $\rho$  and $\olg$, 
since it can also be written as $|d\log\rho|_g \to 1$ when $\rho \to 0$; the interpretation is 
that $-\log \rho$ behaves asymptotically like a distance function for $g$. 

The conformal infinity of a conformally compact metric $(M,g)$ is the conformal class $\frakc(g)$ 
on $\del M$ defined in (\ref{eq:confinf}). This is the beginning of a correspondence between the 
interior Riemannian geometry of $g$ and conformal geometry on the boundary which is a key motivation 
for studying conformally compact metrics, see \cite{FG} for more on this.

\subsubsection*{Normal form and special bdf's} Any conformally compact metric can be put into the normal 
form (\ref{eq:norform}) in a neighborhood of infinity. More specifically, let $(M,g)$ be AH and 
suppose that $h_0$ is any $\calC^{3, \alpha}$  metric on $\del M$ which represents the conformal class $\frakc(g)$. 
Graham and Lee \cite{GL} proved that there exists a unique defining function $x$ for $\del M$ in some 
neighborhood $\calU$ of the boundary which satisfies the two conditions:
\begin{equation}
|d\log x|_g^2 \equiv 1, \quad \mbox{in}\ \ \calU, \qquad  \mbox{and}\quad \left. x^2 g\right|_{T\del M} = h_0.
\label{eq:HJpre} 
\end{equation}
To see this, choose an arbitrary smooth boundary defining function $\rho$ and set $\olg = \rho^2 g$. 
Then $h_0 = e^{2\phi_0} \, \left. \olg \right|_{T\del M}$ for some function $\phi_0$ on $\del M$, and the new 
boundary defining function $x = e^{\phi}\rho$ satisfies (\ref{eq:HJpre}) if and only if $\phi$ satisfies the 
nondegenerate Hamilton-Jacobi equation
\begin{equation}
| d\rho + \rho \, d \phi|_{\olg}^2 \equiv 1\qquad  \Longleftrightarrow 
\qquad 2\, \langle \partial_\rho, \nabla^{\olg} \phi\rangle_{\olg} 
+ \rho \, |\nabla^{\olg} \phi|_{\bar g} = \frac{1- |\nabla^{\olg} \rho|^2_{\olg}}{\rho}  \, .
\label{eq:HJ}
\end{equation}
This is a noncharacteristic first order differential equation which has a unique solution with given boundary 
data $\phi (0, \cdot) = \phi_0$. The corresponding function $x$ is called a special boundary defining function, 
or special bdf. 

It is useful to think of the solution $\phi(x,y)$ of the Hamilton-Jacobi  equation (\ref{eq:HJ}) as the result of 
applying an `extension operator' $\calE$ to the boundary data  $\phi_0$, so we write $\phi = \calE (\phi_0)$. 
\begin{lemma} 
For any sufficiently small ball around the origin, $\calV \subset \calC^{2,\alpha}(\del M)$, there exists a 
collar neighbourhood $\calU$ of $\del M$ in $\overline M$ such that the solution operator 
\[
\calE: \calV \longrightarrow \calC^{2,\alpha}(\overline{\calU})
\] 
is continuous.
\end{lemma}
\begin{proof} 
We shall simply trace through the standard proof of existence for this class of equations to make sure that
the regularity is as stated.

First, rewrite the equation as
\[
F(\rho,y, d\phi) := \del_\rho \phi - \frac12 \, \rho  \, |\nabla^{\olg}\phi|^2_{\olg} + A(\rho,y) = 0,
\]
where $A(\rho,y)$ is the inhomogeneous term on the right in (\ref{eq:HJ}). Thus $F(\rho,y,p,q)$ is $\calC^{2,\alpha}$ 
in $(\rho,y)$ and a quadratic polynomial with coefficients vanishing at $\rho=0$ in $(p,q)$ (which represent 
the components $(\del_\rho \phi, \del_y \phi)$ of $d\phi$). We first find the graph of $d\phi$, parametrized by some auxiliary 
parameters $(s,\eta_1, \ldots, \eta_n)$. Solve the Hamiltonian system with independent variable $s$ and initial conditions
\[
\frac{d\rho}{ds}=F_p,\quad \frac{dy}{ds}=F_q,\quad \frac{dp}{ds} = -F_\rho,\quad \frac{dq}{ds} =-F_y 
\]
with initial data \[
\rho(0,\eta) = 0,\  y(0,\eta) = \eta,\ p(0,\eta) = F(0,\eta,0,0) = A(0,\eta),\  q(0,\eta) = \del_y \phi_0(\eta).
\]
The components of $\nabla F$ are all $\calC^{1,\alpha}$ or better, and the initial condtions 
are also $\calC^{1,\alpha}$, so by standard ODE theory, there is a unique local solution 
$(\rho(s,\eta),y(s,\eta),p(s,\eta),q(s,\eta))$ which is $\calC^{1,\alpha}$. To invert the map $(s,\eta) \mapsto 
(\rho,y)$, simply note that since the Jacobian of this transformation at $\rho = 0$ is the identity, we can apply 
the inverse function theorem. This gives a local $\calC^{1,\alpha}$ inverse, and hence $p$ and $q$ are 
$\calC^{1,\alpha}$ functions of $(\rho,y)$. 
Now solve $\phi_\rho=p(\rho,y)$ by integrating from $\rho = 0$; then $\phi_y=q(\rho,y)$ automatically
holds by the usual Hamiltonian formalism. Therefore, $\nabla \phi(\rho,y)\in \calC^{1,\alpha}$, 
so $\phi \in \calC^{2,\alpha}$ as claimed.  This proof is local in $\rho$, but may be carried out
globally on $\del M$, which gives the precise statement of the theorem. 
\end{proof}

For later reference, now suppose that $x$ and $\wh{x}$ are two special boundary defining functions which
induce boundary metrics $h_0$ and $\wh{h}_0 = e^{2\phi_0}h_0$, respectively and let $\olg = x^2 \, g$. 
Then $\wh{x} = e^{\phi} \, x$ where $\phi(x,y)$ satisfies
\begin{equation}
- 2Ê\, \del_x \phi = x \, | {\nabla}^{\olg}  \phi |^2_{\olg}.
\label{eq:HJsbdf}
\end{equation}
From this it follow that if $\phi_0 \in \calC^{2,\alpha}(\del M)$, then
\begin{equation}
\phi(x,y) = \phi_0(y) - \frac14 |\nabla^{h_0} \phi_0 |^2_{h_0} \, x^2 + \calO(x^{2+\alpha}).
\label{eq:HJsbdf2}
\end{equation}

Now, given $h_0 \in \frakc(g)$ and associated special bdf $x$, define $\olg$  as above, and write
$\olN = {\nabla}^{\olg} x$.  Then, using the exponential map with respect to $\olg$, 
\[
[0,x_0) \times \del M \ni (x,y) \longmapsto \Phi_{h_0}(x,y) := \exp_{y}(x \olN),
\]
defines a diffeomorphism between $[0,x_0)_x \times \del M$ and $\calU$, and also identifies each level set 
$\{x =  \mbox{const.} \}$ with $\del M$. By Gauss' Lemma, 
\begin{equation}
\Phi^*_{h_0}(g) = \frac{dx^2 + h(x)}{x^2},
\label{eq:norform2}
\end{equation}
where $h(x)$ is a family of metrics on $\del M$ which depends on $x \in [0,x_0)$.  
This exhibits the bijective correspondence between elements $h_0 \in \frakc(g)$ and 
special boundary defining functions. 

In this paper, we assume that $\olg$ has a $\mathcal C^{3,\alpha}$ extension up to $\del M$, which
implies that $h(x)$ admits a second order Taylor expansion in powers of $x$, 
\begin{equation}
h(x) = h_0 + h_1 \, x + h_2 \, x^2 + \calO (x^{3}) , 
\label{eq:expansion}
\end{equation}
where the coefficients $h_j$ are symmetric two-tensors on $\del M$; these can be calculated using the formula
\begin{equation}
h_j = \frac{1}{j!} \, \calL_{\olN}^{(j)} \olg \, |_{x=0}.
\label{eq:Liederiv}
\end{equation}
Here $\calL$ is the Lie derivative, and we use that $\calL_{\olN} \, dx^2 = 0$ so $\calL_{\olN}\, \olg = \calL_{\olN} \, h$. 
In particular $h_j$ is $\mathcal C^{3-j}$, for $j=0,1,2$. 

\subsubsection*{Special bdf's and hypersurfaces}
If $x$ is a special bdf and $x' \in (0, x_0)$, then $r = -\log (x/ x')$ is the signed distance function for the 
hypersurface $\Sigma =  \{ r=0 \} = \{ x = x' \}$ near $\del M$. Conversely, if $\Sigma$ is any hypersurface in 
$M$ for which the exponential map from the outward pointing normal bundle $N^+ \Sigma$ to the exterior of 
$\Sigma$ in $M$ is a diffeomorphism, and if $r = \mbox{dist}_g(\cdot, \Sigma)$ (by which we mean the signed distance 
function which is positive outside $\Sigma$), then $x = e^{-r}$ is a special bdf. As above, it induces a metric 
$h_0 \in \frakc(g)$, as the boundary trace of $x^2 g$ via $x^2 g = \olg = dx^2 + h(x)$ where $h(0) = h_0$. 

These simple observations provide a key for what is to follow, namely that because of this bijective correspondence 
between elements in $\frakc(g)$ and outwardly convex hypersurfaces, one may study the geometry of these hypersurfaces 
by methods of conformal geometry. A well-known but never published paper \cite{Ep} by Epstein discusses this 
correspondence in great detail in three-dimensional hyperbolic space. 

\subsubsection*{Poincar\'e-Einstein metrics}
A particularly interesting class of AH manifolds are the Poincar\'e-Einstein (PE) spaces, where the metric $g$ is 
Einstein with $\Ric(g) + ng = 0$. This Einstein condition forces numerous relationships between the coefficient 
tensors $h_j$ in the Graham-Lee normal form of $g$. In particular, for $j < n/2$, $h_{2j+1} = 0$, while each $h_{2j}$ 
is given by a conformally natural partial differential operator of order $2j$ applied to $h_0$. For example,
$h_2 = -P_{h_0}$, where
\[
P_{h_0} = \frac{1}{n-2}\left(\Ric(h_0) - \frac{R_{h_0}}{2(n-1)} \, h_0\right) ,
\]
is the Schouten tensor of $h_0$. 

We do not need the full force of the PE condition in this paper, but shall often work with metrics $g$ which 
are {\it weakly Poincar\'e-Einstein}, in the sense that
\[
h(x) = h_0 + h_2 \, x^2 + \calO(x^{3}), \qquad \mbox{with} \qquad h_2 = - P_{h_0}\, ;
\]
this encompasses many interesting cases

\section{CMC foliations}
We now present some general facts about hypersurface foliations near infinity in conformally compact manifolds where 
each of the leaves has CMC, and also review some familar examples.

\subsubsection*{Foliations and defining functions} Let $(M,g)$ be an AH manifold, $\calU$ a neighborhood of infinity, and 
$\calF = \{\Sigma_\tau\}$, $0 < \tau < \tau_0$, a foliation of $\calU$ by CMC hypersurfaces. We suppose that $\Sigma_\tau 
\to \del M$ as $\tau \searrow 0$, and that for each $\tau$ the exponential map is a diffeomorphism from the outward
pointing normal bundle of $\Sigma_\tau$ to the unbounded component of $M \setminus \Sigma_\tau$. 

\begin{definition} 
We say that $\calF$ is monotone increasing or decreasing if the mean curvature $H(\Sigma_\tau) = H_\tau$ is a monotonically 
strictly increasing or decreasing function of $\tau$. In either of these cases, we say that $\calF$ is a monotone foliation.
\end{definition}

There is no natural choice for the parameter $\tau$ which indexes the leaves. The foliations constructed here are perturbations 
of the level sets $\{x=\e\}$ where $x$ is a particular special bdf, so $\e$ is a reasonable choice of parameter. For 
other foliations near infinity in an AH space, there may not exist any choice of parametrization of the leaves which 
extends smoothly to $\olM$ and is a defining function for $\del M$. This is related to whether the leaves 
converge to $\del M$ more rapidly in some regions than others.

\subsubsection*{Uniqueness}
As already noted, one motivation for finding CMC foliations in the first place is that they might be canonical objects.
We now describe one situation in which this is the case.
\begin{proposition}
Suppose that $\calF$ is a monotone decreasing CMC foliation in one end of a conformally compact manifold $(M,g)$. Then 
$\calF$ is unique amongst all CMC foliations with compact leaves tending to $\del M$ in that end.
\label{pr:unic}
\end{proposition}
\begin{proof}
The uniqueness follows by a simple application of the maximum principle. Let $\calF$ be the given monotone decreasing 
foliation, and let $\calF'$ be any other CMC foliation and $\Sigma'$ any leaf of $\calF'$. Consider the set of leaves 
$\Sigma_\tau$ `outside' $\Sigma'$. There is a maximal value $\tau_1$ for which this is true, and clearly $\Sigma_{\tau_1}$ 
meets $\Sigma'$ tangentially. Similarly, consider the set of leaves $\Sigma_\tau$ which are inside $\Sigma'$, and let 
$\tau_2$ denote the smallest value of $\tau$ for which this is true. Again, $\Sigma_{\tau_2}$ meets $\Sigma'$ tangentially. 
Let $H_1$, $H_2$ and $H$ denote the mean curvatures of $\Sigma_{\tau_1}$, $\Sigma_{\tau_2}$ and $\Sigma'$, respectively.  
On the one hand, by the comparison principle for mean curvature,  since $\Sigma_{\tau_2}$ lies on the mean convex 
side of $\Sigma$ and is tangent to it, we have $H_2 \geq H$. Similarly,  since $\Sigma$ lies on the mean convex side of 
$\Sigma_{\tau_1}$ and is tangent to it, we have $H \geq H_1$. Since $\calF$ is monotone decreasing, $H_1 \geq H_2$, with 
equality if and only if $\Sigma_{\tau_1} = \Sigma_{\tau_2}$. Therefore, we conclude that $H_1 =H_2$ and  
$\Sigma_{\tau_1} = \Sigma_{\tau_2} = \Sigma'$, as required.
\end{proof}

\subsubsection*{Basic examples}
We present now three basic examples of CMC foliations near infinity.

\medskip

\noindent {\bf 1. Geodesic balls} Let $(M,g)$ be the hyperbolic space $\HH^{n+1}$. Fix the origin $o$, and let $S_R$ 
denote the geodesic ball of radius $R > 0$ centered at $o$. We can recover the conformally compact form of the metric 
from the geodesic polar  coordinate expression $g = dr^2 + \sinh^2r\, d\theta^2$ by setting $x = 2 e^{-r}$, so that
\[
g = \frac{dx^2 + (1 - \frac{x^2}{4})^2\, d\theta^2}{x^2}.
\]
Note that this is the special bdf corresponding to the standard round metric $d\theta^2$ in the conformal infinity of $g$. 

We set $\Sigma_\e = S_R$ where $R = -\log(\e/2)$. Each such leaf is totally umbilic, with  second fundamental form
\[
\II(\Sigma_\e) = \e^{-2}\left(1 - \frac{\e^4}{16}\right)\, d\theta^2, 
\]
and all principal curvatures equal to $(4+ \e^2 )/(4- \e^2 )$. Hence this is a Weingarten  foliation no matter what 
the function $f$. The mean curvature
\[
H(\Sigma_\e) = n \, \left(\frac{4+ \e^2 }{4- \e^2 }\right),
\]
is monotone increasing in $\e$, with
\[
\lim_{\e \searrow 0} H(\Sigma_\e) = n, \qquad \lim_{\e \nearrow 2} H(\Sigma_\e) = \infty.
\]
The latter limit corresponds to geodesic spheres of radius tending to zero. There is a family  of such foliations 
obtained by shifting the center of the balls, so uniqueness clearly fails.

\medskip

\noindent {\bf 2. Equidistant hypersurfaces}  The second example concerns the  family of hypersurfaces in the warped 
product $M = \RR \times Y$, with metric $g = dt^2 + \cosh^2 t\, h$, which are equidistant from the `core' $\{t=0\}$.  
Here $(Y,h)$ is any compact $n$-manifold. When $(Y,h)$ is Einstein, with
Ricci curvature $-(n-1)h$, then $(M,g)$ is also Einstein with Ricci curvature $-ng$. In particular, if $\dim Y = 2$ and 
$(Y,h)$ is hyperbolic, then $M$ is called a Fuchsian hyperbolic $3$-manifold. However, $(M,g)$ is always conformally 
compact,  as can be seen by setting $x = 2e^{-t}$ once again to get 
\[
g = \frac{dx^2 + (1 + \frac{x^2}{4})^2 h}{x^2}.
\]
Thus $x$ is a special bdf corresponding to the metric $h$ in the conformal infinity of $g$ at $t=\infty$.
(The situation at the end $t=-\infty$ is analogous.) 

The surfaces $\Sigma_\e = \{t = -\log(\e/2)\}$ foliate this end, and have second fundamental form
\[
\II(\Sigma_\e) = \e^{-2} \, \left(1 - \frac{\e^4}{16}\right)\, h,
\]
so these are again totally umbilic with principal curvatures $(4- \e^2)/(4+\e^2)$. We have 
\[
H(\Sigma_\e) = n \, \left(\frac{4- \e^2}{4+ \e^2}\right),
\]
which is monotone decreasing in $\e$, with
\[
\lim_{\e \searrow 0} H(\Sigma_\e) = n, \qquad \lim_{\e \nearrow 2} H(\Sigma_\e) = 0,
\]
the latter limit corresponding to the central core at $t=0$. 

\medskip

\noindent {\bf 3. Horospheres} The final example is of a horospherical foliation in a conformally compact
 manifold $(M,g)$ with end diffeomorphic to the product $\RR^+ \times T^n$ and with warped product metric $g = 
dt^2 + e^{2t}\, h$, where  $h$ is the flat metric on the torus $T^n$. Now set $x = e^{-t}$, so $g = (dx^2 + h)/x^2$, 
and if $\Sigma_\e = \{t = -\log \e\}$, then
\[
\II (\Sigma_\e) =  x^{-2}h
\]
has all principal curvatures equal to $1$, and hence
\[
H(\Sigma_\e) = n
\]
for all $\e$. In particular, this foliation is neither (strictly) monotone increasing or decreasing. 

\section{Geometric calculations}
We now present a series of calculations for the second fundamental forms and mean curvatures of the level sets of a 
special bdf. The first step is to compute these for a given special bdf $x$, and after that we examine the effect of 
changing the conformal representative on the boundary.

\subsubsection*{Second fundamental form  of level sets}
Suppose that $(M,g)$ is an AH metric and fix $h_0 \in \frakc(g)$ and the corresponding special bdf $x$.  
Set $\olg = x^2 g$ and $\olN = \nabla^{\olg} x$; we also set $N = x^{-1} \,  \nabla^g x$. 
We now calculate the second fundamental form and mean curvature of the level sets $\{x = \mbox{const.}\}$. 

To begin this calculation, recall the standard formula for the second fundamental form of the level sets 
$\{x = \mbox{const.}\}$
\[
\II = -\frac12 \calL_N g,
\]
which holds because $N$ is the gradient of the distance function from each level set. We recall two standard facts:
\begin{equation}
\calL_W(f \kappa) = f \calL_W \kappa + (Wf)\, \kappa, \qquad \mbox{and} \qquad 
\calL_{(fW)}\, \kappa = f \calL_W \kappa + df \circ \iota_W \kappa,
\label{eq:ldi}
\end{equation}
here $\beta \circ \gamma = \beta \otimes \gamma + \gamma \otimes \beta$. 

Since $N = x \, \olN$ and $g = x^{-2} \, \bar g$, we compute
\begin{equation*}
\begin{split}
\calL_N g = \calL_{x\olN}(x^{-2}\olg) & = x \, \calL_{\olN}(x^{-2}\olg) + dx \circ \iota_{\olN}(x^{-2}\olg) \\
 & = x \left(x^{-2} \, \calL_{\olN}\olg - 2x^{-3}\olg\right) + 2 \, x^{-2}dx^2 \\
 & = x^{-2}\left(x \,  \calL_{\olN}\olg - 2\olg + 2 \, dx^2\right) \\
 & = x^{-2}\left( x\, \calL_{\olN} \, h - 2 \, h \right),
 \end{split}
\end{equation*}
since $\calL_{\olN} \, dx^2 = 0$. Write $\calL_{\olN} h = \del_x h$ for simplicity; then the second fundamental 
form of  the level sets $\{x= \mbox{const.}\}$ is given by
\begin{equation}
\II = \frac12 x^{-2} (2 \, h - x\del_x \, h ).
\label{eq:sff}
\end{equation}

The metric induced on $\{x = \mbox{const.}\}$  is given by $x^{-2} \, h$ and hence the mean curvature of this 
hypersurface is given by 
\begin{equation}
H = \frac12\, \tr^{x^{-2} h}\left( x^{-2}( 2\,  h - x \, \del_x h  )\right)  =  n - \frac12\, \tr^h \left( x \, \del_x h \right).
\label{eq:H1}
\end{equation}

The expansion for $h(x)$ then yields
\begin{equation}
H = n - \frac12 \left(\tr^{h_0}h_1\right) \, x - \left( \tr^{h_0}h_2 - \frac12 \, \|h_1\|_{h_0}^2\right) \, x^2 + 
\calO(x^{3}).
\label{eq:H2}
\end{equation}
The value of this function at any point is the mean curvature (with respect to the metric induced by $g$) 
of the level set of the special bdf $x$ through that point. In terms of the notation in (\ref{eq:k1k2}) we can write 
\[
H = n - \kappa_{1} \, x - \kappa_{2}  \, x^2 + \calO(x^{3}).
\]

\subsubsection*{Effect of conformal changes}
We next study the effect on various geometric quantities of changing the representative $h_0$ within the 
conformal class $\frakc(g)$. 

Set $\wh{h}_0 = e^{2\phi_0}h_0$.  We now have the special bdf $\wh{x} = e^{\phi} \, x$ associated to $\hat h_0$ and
the metric $\wh{g} = \wh{x}^{\,2}g = e^{2\phi}\olg$, and can write
\[
\wh{g} = d\wh{x}^{\,2} + \wh{h}(\wh{x}), \qquad \wh{h}(\wh{x}) = \wh{h}_0 +  \wh{h}_1 \, \wh{x} +  
\wh{h}_2 \, \wh{x}^{\,2} + {\mathcal O} (\wh{x}^{3}). 
\]
Clearly
\[
\wh{h}_0 = \left. \wh{g} \right|_{T\del M} = e^{2\phi_0}h_0\, ;
\]
the higher terms in the expansion for $h$ may be computed using 
\[
\widehat{h}_j = \left. \frac{1}{j!} \calL_{\widehat{N}}^{(j)}\, \wh{g} \right|_{\del M},
\]
where
\[
\wh{N} = \wh{x}^{-1} \, \nabla^g \, \wh{x} = x\, e^{-\phi} \,{\nabla}^{\olg} (e^{\phi}x) = x \, (\olN + x \, {\nabla}^{\olg}
\phi) := x\, Z, \qquad \mbox{where}\ \  Z = \olN + x \, {\nabla}^{\olg} \phi \, .
\]
In particular, to calculate the second fundamental form of  any level set $\{\wh{x} = \e \}$, we first compute 
\begin{equation*}
\begin{array}{rlll}
\calL_{\wh{N}} \, g & = & \calL_{xZ}(x^{-2}\olg) \\[3mm]
& = & x \, \calL_Z(x^{-2}\olg) + dx \circ \iota_Z(x^{-2}\olg) \\[3mm]
& = & x^{-1} \, \calL_Z \olg - 2x^{-2} (Zx) \olg + x^{-2}dx \circ \iota_Z \olg.
\end{array}
\end{equation*}
Using the decomposition of $Z$, this splits further as
\begin{equation*}
\calL_{\wh{N}} g = x^{-2} \, \left( x\, \del_x \, h - 2 \, h \right) + \calL_{{\nabla}^{\olg}\phi}\olg + 
2x^{-1}dx \circ d\phi - 2 \, x^{-1} \, \del_x \phi \,  \olg.
\end{equation*}
Since 
\[
\calL_{{\nabla}^{\olg} \phi} \olg = 2 \, \Hess^{\olg} (\phi),
\]
the second fundamental form of each level set $\{\wh{x} = \e  \}$  is 
\begin{equation}
\II (\phi_0) = \left( \frac12 \, x^{-2} \, \left( 2 \, h - x\, \del_x h \right) - x^{-1}\,  dx \circ d\phi - 
\Hess^{\bar g} (\phi) -\frac12  \, |{\nabla}^{\olg} \phi|^2_{\olg} \, \olg \, \right) \big|_{x \, e^\phi  = \e}\, .
\label{eq:sff1}
\end{equation}

The mean curvature of the level set $\{\wh{x} =  \e \}$ is the trace of $\II (\phi_0)$ with respect to the induced 
metric $\wh{x}^{\, -2} \wh{h}$ on this hypersurface. However, since there is no $d\wh{x}^{\, 2}$ component, 
we can simply take the trace with respect to $g$ directly; using (\ref{eq:HJsbdf}), this gives
\begin{equation*}
\begin{array}{rlll}
H(\phi_0) &= \left( \tr^{g} \II(\phi_0) \right) |_{x \, e^\phi = \e} = \left( \tr^{x^{-2}\olg} 
\II(\phi_0) \, \right) |_{x \, e^\phi = \e} \\[3mm]
& = \displaystyle \left( \frac 12 \, \tr^h  \left(  2 \, h - x \, \partial_x h  \right)  
- x^2 \left(\Delta_{\olg} \phi +  \frac{n-1}{2} \, 
| {\nabla}^{\olg} \phi|^2_{\olg} \right)\, \right) \big|_{x \, e^\phi = \e}. 
\end{array}
\end{equation*}
This proves the 
\begin{proposition}
The mean curvature of the hypersurface $\{\wh{x} =  \e \}$ is given by 
\begin{equation}
H(\phi_0,\e)  :=   n -  \e^2 \, \left( e^{-2\phi} \, \left( \frac 1{2} x^{-1}\, \tr^h  \partial_x h  + 
\Delta_{\olg} \phi + \frac{n-1}{2} |{\nabla}^{\olg}  \phi|^2_{\olg}  \right)\right) \big|_{x \, e^{\phi}= \e }.
\label{eq:Hhate}
\end{equation}
\label{pr:hconfch}
\end{proposition}

These expressions for the second fundamental form and mean curvature conceal the fact that $\phi$ is actually a 
nonlocal function of $\phi_0$, as is the operation of restricting to the level set $\wh{x} = \e$.

Expanding $H(\phi_0,\e)$ in powers of $\e$ shows how the coefficients in (\ref{eq:H2}) are affected by a change of
representative $h_0$ of the conformal class $\frakc(g)$. Indeed, using (\ref{eq:HJsbdf2}) 
\[
\Delta_{\olg}\phi = \Delta_{h_0}\phi_0 - \frac{1}{2} \, |\nabla^{h_0}  \phi_0|^2_{h_0} + 
{\mathcal O}\left(x^\alpha \right),
\]
which shows that, in contrast to (\ref{eq:H2}), 
\begin{equation}
\begin{split}
H(\phi_0,\e)  & =   n -  \frac 12 \, e^{-\phi_0} \, \tr^{h_0}  h_1\, \wh{x} \hfill \\ & 
- e^{-2\phi_0} \, \left(   \tr^{h_0} h_2 - \frac12 \, \|h_1\|_{h_0}^2 +  
\Delta_{h_0} \phi_0 + \frac{n-2}{2} |{\nabla}^{h_0} \phi_0|^2_{h_0}  \right) \,  \wh{x}^2 + {\mathcal O} (\wh{x}^{3}).
\end{split}
\label{eq:ccpa}
\end{equation}

This shows how the functions $\kappa_{1}$ and $\kappa_{2}$ in (\ref{eq:k1k2}) transform under a conformal change of metric. 
Indeed, if $\kappa_{1}, \kappa_{2}$ and $\tilde \kappa_{1}, \tilde \kappa_{2}$ are the functions corresponding to the 
conformal representatives $h_0$ and $\tilde{h}_0 = e^{2\phi_0}h_0$, respectively, then (\ref{eq:ccpa}) shows that
\begin{eqnarray}
\tilde \kappa_{1}   & = &  e^{-\phi_{0}} \, \kappa_{1}  \, , \label{tildek1}  \\[2mm]
\tilde \kappa_{2}   & = & e^{-2\phi_0} \, \left(   \kappa_{2}  + \Delta_{h_0} \phi_0 + 
\frac{n-2}{2} |{\nabla}^{h_0} \phi_0|^2_{h_0}  \right) .\label{tildek2}
\end{eqnarray}
In particular, the conditions $\tr^{h_{0}} h_{1} \equiv 0$, $\tr^{h_{0}} h_{1}> 0$ or $\tr^{h_{0}} h_{1} < 0$ 
do not depend on the choice of the  conformal representative of $ \frakc(g)$. 

\section{Existence of CMC foliations}
We now establish three separate existence theorems, in order of increasing analytic difficulty, corresponding to 
whether $\kappa_1 = \tr^{h_0}h_1$ vanishes identically, or is everywhere positive or everywhere negative,  
espectively. After that we prove that each of these families of CMC surfaces fit together in a foliation.

\subsubsection*{Case 1: $\tr^{h_0}h_1 \equiv 0$} According to (\ref{eq:H2}) 
\[
H (0, \e)  = n -  \kappa_2 \, \e^2 + {\mathcal O} (\e^3), 
\]
where $\kappa_2 \in \calC^{1,\alpha} (\del M)$. Let us assume for the time being that $\kappa_2 = \mbox{const}$. 
We seek, for each $\e > 0$, a function $\phi_0 = \phi_0(\e)$ on $\del M$ so that
\[
 H (\phi_0(\e),\e) = n -  \kappa_2 \, \e^2, 
\]
By (\ref{eq:Hhate}), this is equivalent to 
\begin{equation}
\calN (\phi_0 ,\e) - \kappa_2  = 0,
\label{eq:eqHcase1}
\end{equation}
where, by definition,
\begin{equation}
\calN(\phi_0,\e) := e^{-2\phi} \, \left( \frac 1{2x} \, \tr^h  \partial_x h   +  \Delta_{\olg} \phi +  
\frac{n-1}{2} \, |{\nabla}^{\olg} \phi|^2_{\olg}  \right) \big|_{x \, e^{\phi}= \e }.
\label{eq:eqcase1}
\end{equation}
Since $\kappa_1 = 0$, it follows that $\tr^h \del_x h = 2 \, \kappa_2 x + \calO (x^2)$, so $\calN$ is $\calC^1$ 
in $\e$ up to $\e=0$ and $\calN(0,0) = \kappa_2$.  
 
\begin{theorem}
If $\kappa_1 = 0$, $\kappa_2 = \mbox{const.}$ and $\Delta_{h_0} -2 \, \kappa_2$ is invertible, then for 
each small $\e > 0$, there is a unique solution to (\ref{eq:eqHcase1}) close to $0$. The hypersurfaces $x \, e^{\phi } = \e$ 
constitute, as $\e$ varies, a monotone CMC foliation near $\del M$. This foliation is unique amongst all 
possible foliations if $\kappa_2 > 0$. 
\end{theorem}
\begin{proof} The proof is a direct consequence of the implicit function theorem, but to see this we must compute
the linearization of $\calN$ at $\phi_0 = 0$. Rewrite $\calN$ as the composition of three operations:
the restriction $R(\phi_0,\e)$ to the hypersurface $\{x e^\phi = \e\}$, the nonlinear partial differential
operator 
\[
\phi \longmapsto e^{-2\phi} \, \left(\frac12 x^{-1}\tr^{h}\del_x h + \Delta_{\olg}\phi + \frac{n-1}{2}|\nabla^{\olg}\phi|^2_{\olg}\right),
\]
and the extension operator $\phi = \calE(\phi_0)$. By the chain rule, 
\[
\left. D_1 \calN \right|_{(0,\e)} = \left. D_1 R\right|_{(0,\e)} \circ \left(\Delta_{\olg} - x^{-1} \, \tr^{h}(\del_x h)\right) 
\circ D\calE_{0}(\psi_0).
\]
Here $\psi(x,y) := D\calE_0 \psi_0$ is the solution of the linearization of the Hamilton-Jacobi equation (\ref{eq:HJsbdf})
with initial condition $\psi_0$, so $\del_x \psi = 0$ and $\psi(0,y) = \psi_0$ and hence $\psi(x,y) = \psi_0(y)$. 
The differential of the restriction operator is not so easy to compute in general, but is simply the restriction to 
$\del M$ when $\e=0$. Altogether then,
\begin{equation}
D_1\calN_{(0,0)} \psi_0 = (\Delta_{h_0} - 2\kappa_2) \psi_0,
\label{eq:linN}
\end{equation}
which by assumption is invertible. Hence there exists a smooth function  $\e \to \phi_0(\e)$, with $\phi_0(0) =0$ 
and $\calN(\phi_0(\e),\e) - \kappa_2 =0$ for $0 \leq \e < \e_0$. The proof that these hypersurfaces 
form a foliation is deferred until the end of the section. The uniqueness statement follows because
$\kappa_2 > 0$ implies that the foliation is monotone decreasing so that Proposition 3.1 applies. \end{proof}

The transformation rule (\ref{tildek2} provides a way to reduce the general case where $\kappa_2$ is a 
function to this special case where $\kappa_2$ is constant. Indeed, suppose we have found a function 
$\bar \phi_0$ such that 
\begin{equation}
e^{-2 \bar \phi_0} \, \left( \Delta_{h_0} \bar \phi_0 + \frac{n-2}{2} \, |\nabla^{h_0} \bar \phi_0|^2_{h_0} +
\kappa_2 \right) - \bar \kappa_2  = 0,
\label{eq:YTipe}
\end{equation}
where $\bar \kappa_2$ is constant; then the term corresponding to $\kappa_2$ for the new metric 
$\bar h_0  = e^{2\bar \phi_0} \, h_0$ is this constant $\bar \kappa_2$. 

According to the result above, the existence of CMC foliations reduces to the existence of non degenerate solutions of 
(\ref{eq:YTipe}). We discuss this issue of solvability briefly now.  For the sake of simplicity, let us focus on 
the case $n \geq 3$. When $\kappa_2$ is an arbitrary smooth function, there may or may not be a solution to the
equation (\ref{eq:YTipe}).  We claim, however, that there is a solution, which is in fact unique, if the least 
eigenvalue $\lambda_1$ of the generalized conformal Laplacian 
\[
\mathbb L_{h_{0}} : = - \left( \Delta_{h_0} + \frac{n-2}{2}  \,  \kappa_2 \right)
\]
is negative. The proof is an adaptation of that for an analogous result for the Yamabe equation. To make the analogy
more clear, set $\bar \phi_0 = \frac{2}{n-2} \, \log u_0$, which transforms (\ref{eq:YTipe}) into the more 
familiar-looking equation
\begin{equation}
\mathbb L_{h_0} u + \frac{n-2}{2} \, \bar \kappa_2  \, u_0^{\frac{n+2}{n-2}} = 0.
\label{eq:YTE}
\end{equation}
To see that the sign of this least eigenvalue is independent of choice of conformal representative, we proceed
as follows. Let $\bar{h}_{0} = u_{0}^{\frac{4}{n-2}} \, h_{0}$ be two conformally related metrics. By direct computation
we deduce the general formula
\[
u_{0}^{\frac{n+2}{n-2}}  \, \Delta_{\bar h_{0}}  w =   \Delta_{h_{0}} \, (u_{0} \, w) - (\Delta_{h_{0}} u_{0} ) \,  w,
\]
so combining this with (\ref{tildek2}) gives
\begin{equation}
\mathbb L_{\bar h_{0}} w =  u_{0}^{- \frac{n+2}{n-2}} \, \mathbb  L_{h_{0}} (u_{0} \, w) \, , 
\label{eq:3lap}
\end{equation}
and hence
\[
\int_{\del M}w \, \mathbb  L_{\bar h_{0}}  w \, dV_{\bar h_{0}} = \int_{\del M} (uw)\, \mathbb L_{h_{0}} (u_{0} \,w) \, 
dV_{h_{0}}\ ;
\]
this shows that the sign of $\lambda_1(\mathbb  L_{h_{0}})$ is the same as that of $\lambda_1(\mathbb  L_{\bar{h}_0})$. 

Now, suppose that $\lambda_{1}(\mathbb  L_{h_0}) < 0$ and let $\varphi_1$ be the corresponding eigenfunction. 
If $u_0 >0$ is a solution to (\ref{eq:YTE}), then multiplying this equation by $\varphi_{1}$ and integrating yields
\[
\lambda_1 \int_{\del M}  \varphi_{1} \, u_{0} \, dV_{h_{0}} + \frac{n-2}{2} \, \bar \kappa_{2}  \, 
\int_{\partial M}  \varphi_{1} \, u_{0}^{\frac{n+2}{n-2}} \, dV_{h_{0}} =0
\]
Since both $\varphi_{1}$ and $u_{0}$ are positive, $\lambda_{1}$ and $\bar \kappa_{2}$ must have opposite signs, 
and so if there is a solution in this case, then necessarily $\bar {\kappa}_ 2 > 0$. 

To produce a solution, fix $\bar \kappa_{2} >0$ and for each $1 <p < \frac{n+2}{n-2}$, minimize the functional
\[
E_{p} (u) = \frac{1}{2} \int_{\del M} \left( |\nabla^{h_0} u|^2_{h_0} - \frac{n-2}{2} \,  \kappa_2 \, u^2  \right)  \, dV_{h_0} + \frac{n-2}{2 (p+1)} \bar \kappa_{2}  \int_{\del M} |u|^{p+1} \, dV_{h_0}. 
\]
The existence of a positive smooth minimizer $u_{p}$ is classical and $u_{p}$ satisfies the Euler-Lagrange
equation 
\[
\Delta_{h_{0}} u_{p} + \frac{n-2}{2} \kappa_{2} \, u_{p} - \frac{n-2}{2} \, \bar \kappa_{2} \, u^{p}_{p} =0.
\]

Next, we obtain an a priori estimate for the sup of $u_p$ which is independent of $p$. Let $y_{p} \in \del M$
be the point where $u_{p}$ achieves its maximum. Then $\Delta_{h_{0}} u_{p}(y_p) \leq 0$, and hence
\[
\kappa_{2}(y_p) \, u_{p}(y_p) \geq \bar \kappa_{2} \, u^{p}_{p}(y_p),
\]
which implies the uniform bound
\[
\bar \kappa_{2}  \, \|u_{p}\|_{L^\infty}^{p-1} \leq \| \kappa_{2}\|_{L^\infty}.
\]
Using this uniform bound, standard elliptic estimates and the Arzela-Ascoli theorem, we can take the limit of a 
subsequence as $p \nearrow \frac{n+2}{n-2}$, and this gives a smooth positive solution of (\ref{eq:YTE}). 

To prove uniqueness of this solution, assume that $u_{0}$ and $v_{0}$ are both
positive solutions of (\ref{eq:YTE}) (with the same value of $\bar \kappa_{2} > 0$), and define
\[
\psi : = \frac{v_{0}}{u_{0}} .
\]
Then using (\ref{eq:3lap}), we compute that
\[
\Delta_{\bar h_{0}} \psi + \frac{n-2}{2}\, \bar \kappa_{2} \, \left(  \psi - \psi^{\frac{n+2}{n-2}} \right) =0, 
\]
where $\bar h_{0} = u_{0}^{\frac{4}{n-2}} \, h_{0}$. At the point where $\psi$ attains its supremum, $\Delta_{\bar h_{0}} 
\psi \leq 0$, so as above, we conclude that $\psi \leq 1$ everywhere. Similarly, considering the point where 
$\psi$ attains its infimum, we conclude that $ \psi \geq 1$ everywhere. Hence $\psi \equiv 1$, which proves
uniqueness. 

Finally, the linearization of (\ref{eq:YTE}) at $u_0$ is equal to
\[
\Delta_{{h_{0}}} +\frac{n-2}{2} \, \kappa_{2} - \frac{n+2}{2} \,\bar \kappa_{2} \,  u_{0}^\frac{4}{n-2},
\]
and by (\ref{eq:3lap}) and (\ref{eq:YTE}), 
\[
\left( \Delta_{h_{0}} + \frac{n-2}{2} \kappa_{2} - \frac{n+2}{2} \, \bar \kappa_{2} \, 
u_{0}^\frac{4}{n-2} \right) ( u_{0} \, w)  = u_{0}^{\frac{n+2}{n-2}} \, \left( \Delta_{\bar h_{0}} - 
2 \, \bar \kappa_{2} \right) \, w.
\]
Since $\bar \kappa_2 > 0$, $\Delta_{\tilde h_0} - 2 \bar \kappa_2$ and hence this linearization too
must be invertible. 

A similar argument can be made when $n=2$, under the assumption that
\[
\int_{\partial M} \kappa_{2} \, dV_{h_{0}}  >0; 
\]
we leave the details to the reader.  

For weakly Poincar\'e-Einstein metrics, $h_1 =0$ and $h_2 = - P_{h_0}$, so that
\[
\kappa_2 = \tr^{h_0} h_2  = -\frac{1}{2(n-1)}\, R_{h_0}. 
\]
Equation (\ref{eq:YTE}) then becomes
\[
\Delta_{h_0} u_0 - \frac{n-2}{4(n-1)} \, R_{h_0} \,  u_0 - \frac{n-2}{2} \, \bar \kappa_2 \,  u_0^{\frac{n+2}{n-2}} = 0,
\]
which is exactly the Yamabe equation. In this case, we known that there is always at least one solution $u_0 > 0$ with 
$\bar \kappa_2$ constant \cite{LP}. The sign of the value of the least eigenvalue of the conformal Laplacian determines the sign of $\bar \kappa_2$.  In particular, if this least eigenvalue is negative,  there is a monotone decreasing CMC foliation determined by this construction, and this is unique amongst all  possible foliations. On the other hand, if the least eigenvalue of the conformal Laplacian is positive, then to each nondegenerate  solution there exists a foliation; different nondegenerate constant scalar curvature metrics correspond to distinct foliations.

\subsubsection*{The case $\tr^{h_0}h_1  > 0$} 
The next case to consider is when $\tr^{h_0}h_1$ is everywhere positive.  According to (\ref{eq:H2}), 
\[
H (0, \e)  = n - \e \, \kappa_1 + {\mathcal O} (\e^2), 
\] 
where $\kappa_1 \in \calC^{2,\alpha}(\partial M)$. By (\ref{eq:ccpa}), exchanging $h_0$ by a conformal multiple, 
we may as well assume that $\kappa_1 \equiv 1$. 

We seek, for each $\e > 0$, a function $\phi_0 = \phi_0(\e)$ on $\del M$ so that
\[
H (\phi_0(\e),\e) = n - \e , 
\]
Using  (\ref{eq:Hhate}), this is equivalent to solving
\begin{equation}
\tilde \calN(\phi_0,\e) := e^{-\phi} \, \left( \frac1{2} \, \tr^h  \partial_x h   +  
x \, \left( \Delta_{\olg} \phi + \frac{n-1}{2} \, |{\nabla}^{\olg} \phi|^2_{\olg} \right) \right) |_{x \, e^{\phi}= \e } = 1.
\label{eq:eqcase2}
\end{equation}
The linearization of this equation at $\phi_0 = 0$ (and $\e > 0$) is
\[
L_\e : = \e \, \Delta_{h(\e)} -  \frac12 \tr^{h} \partial_x h (\e).
\]
Note that $\frac{1}{2}\, \tr^{h} \partial_x h (\e) = 1 +{\mathcal O}(\e)$.

Define the function spaces $\calC^{k,\alpha}_\e$ to be rescaled H\"older spaces, where every $\del_y$ is accompanied 
by a factor $\sqrt{\e}$. For example,
\[
\| u \|_{{\mathcal C}^{0,\alpha}_\e}  = Ê\sup |u| +  \sup_{y \neq y'}(\sqrt{\e})^\alpha 
\frac{|u(y)-u(y')|}{d(y,y')^\alpha}.
\]
Clearly 
\[
L_\e : \calC^{2,\alpha}_\e \longrightarrow \calC^{0,\alpha}_\e
\]
is bounded independently of $\e$. We claim that the inverse is also bounded uniformly in $\e$, provided $\e$ 
is small enough. In other words, there is some constant $c>0$  independent of $\e$ such that
\begin{equation}
\| u \|_{{\mathcal C}^{2,\alpha}_\e}  \leq c \, \| L_\e \, u||_{{\mathcal C}^{0,\alpha}_\e}.
\label{eq:invunifep}
\end{equation}

To prove this, we rephrase the problem. Define $\tilde{h}(\e) = h(\e)/\e$ and $\tilde L_\e := 
\Delta_{\tilde{h}(\e)} -  \frac{1}{2}\, \tr^{h} \partial_x h (\e)$. Then the spaces $\calC^{2,\alpha}_\e$ are 
simply the standard H\"older spaces with respect to this rescaled metric, and (\ref{eq:invunifep}) is equivalent to
\[
\| u \|_{{\mathcal C}^{2,\alpha}}  \leq C \, \| \tilde L_\e \, u \|_{{\mathcal C}^{0,\alpha}},
\]
on $(\del M,\tilde{h}(\e))$, where $C$ is independent of $\e$. This, in turn, follows from a simple 
scaling argument. If it were to fail, there 
would exist a sequence $\e_j \to 0$ and corresponding functions $u_j\in \calC^{2,\alpha}(M,\tilde{h}(\e))$  
for which $||u_j||_{{\mathcal C}^{2,\alpha}} = 1$, but such that $\|\tilde L_\e \, u_j \|_{{\mathcal C}^{0,\alpha}}
\to 0$. Choose normal coordinates centered at a point $p_j \in M$ where the maximum of $u_j$ occurs. These  
exist on balls of radius $C/\e_j$. The sequence of metrics $\tilde{h}(\e_j)$ converges uniformly on compact 
sets to the Euclidean metric. Passing to a subsequence, the $u_j$ also converge uniformly on compact sets, 
and in the limit we obtain a function $u$  defined on all of $\RR^n$ which satisfies $(\Delta_{\RR^n} - 1) \, u = 0$,
$\sup \, |u| = 1$. This is clearly impossible, hence we have proved the validity of (\ref{eq:invunifep}).

Now write 
\[
L_\e^{-1}\, \left( \tilde \calN (\phi_0,\e) -1 \right) = \phi_0 - J(\phi_0,\e),
\]
where $J$ is a smooth map from $\calC^{2,\alpha}_\e$ to itself, depending smoothly on $\e$, such that
$J(0, \e) = {\calO}(\e)$ and $\left. D_{\phi_0} J\right|_{(0,\e)} = 0$. Note that $J(\phi_0,\e)$ is affine 
in the second partial derivatives of $\phi_0$, which is important.  

The equation to solve, therefore, takes the form
\begin{equation}
\phi_0 = J(\phi_0,\e) .
\label{eq:sdsd}
\end{equation}
Just as before, we can find a solution of this equation in a ball 
of radius $A\e$ in $\calC^{2,\alpha}_\e$, for $A$ sufficiently large.

The solution $\phi_0(\e)$ seems to become increasingly less regular as $\e$ decreases. The fact that
its regularity is controlled uniformly as $\e \searrow 0$ follows from the uniform boundedness of 
$J(0, \e)$ in ${\mathcal C}^{\infty}$ topology (and not in ${\mathcal C}^{\infty}_\e$ topology). More 
precisely, if $X$ be any vector field on $\del M$, then applying $V$ to (\ref{eq:sdsd})  yields a 
linear inhomogeneous elliptic equation for $X \, \phi_0$. There is again a unique solution  with 
norm in $\calC^{2,\alpha}_\e$ bounded by $A' \e$, and by approximating by difference quotients, 
this must be $X \, \phi_0$. Continuing in this way proves that $\phi_0(\e) \in 
\calC^{k,\alpha}(\partial M, h_0)$ for all $k \geq 0$ uniformly as $\e \to 0$. 

\subsubsection*{The case $\tr^{h_0}h_1  <  0$}
The final case, when $\tr^{h_0}h_1$ is everywhere negative, is harder than the previous cases due to a 
resonance phenomenon. It is now necessary to assume that the conformal compactification of $g$ is
$\calC^\infty$. As before, after a preliminary conformal change, we can assume that $\kappa_2 = -2$ 
and so the equation to solve is 
\begin{equation}
\tilde \calN (\phi_0, \e) = -1,
\label{eq:eqcase3}
\end{equation}
where $\tilde{\calN}$ is the same operator as before. The linearization at $\phi_0 = 0$ is again 
$\e \, \Delta_{h(\e)} - \frac12 \tr^{h}\del_x h (\e)$. Since $\frac{1}{2}\, \tr^{h} \del_x h (\e) = 
- 1+ \calO(\e)$, this operator is not invertible for infinitely many values of $\e$ converging to $0$,
so the proof must be handled differently. 

There are two steps in this analysis. First, for any fixed $q \in {\mathbb N}$, we construct a sequence 
of improved approximate solutions $\phi_{0, \e}^{(q)}$ via a simple iteration; this sequence is chosen so 
that $\calN (\phi_{0, \e}^{(q)}, \e) + 1 = \calO(\e^q)$. Second, given any $p > 0$, we produce a sequence of 
disjoint intervals $J_j$ approaching $0$, the union of which has density $1$ at $0$, and such that if 
$\e \in J = \cup J_j$, then 
\[
\| (\tilde L_\e^{(q)} )^{-1} \|_{\calC^{0,\alpha}_\e \to \calC^{2,\alpha}_\e} \leq c_p \, \e^{-p}.
\]
where $\tilde L_\e^{(q)}$ is the linearization at $\phi_{0, \e}^{(q)}$. Using these two results, the argument 
proceeds as before and gives a solution of (\ref{eq:eqcase3}), at least when $\e \in J = \cup J_j$. 

\medskip
\noindent {\bf Improved approximate solution}
We seek a sequence of functions $\phi_{\e}^{(q)}$ satisfying the equation to any specified order in $\e$.  
To this aim, rewrite (\ref{eq:eqcase3}) as
\[
\phi_0 =   e^{-\phi} \, \left( \frac 1{2} \, \tr^h  \partial_x h - e^{\phi} + \phi_0 \, e^{\phi}  + 
x \, \left( \Delta_{\olg} \phi + \frac{n-1}{2} \, |{\nabla}^{\olg} \phi|^2_{\olg} \right) \right) |_{x \, e^{\phi}= \e } .
\] 
Now define the sequence by the recursive relation
\[
\phi_{0, \e}^{(q+1)} =   e^{-\phi_\e^{(q)}} \, \left( \frac 1{2} \, \tr^h  \partial_x h - e^{\phi_\e^{(q)}} +  
e^{\phi_\e^{(q)}}  \, \phi_{0,\e}^{(q)}  +  x \, \left( \Delta_{\olg} \phi_\e^{(q)} + \frac{n-1}{2} \, 
|{\nabla}^{\olg} \phi_\e^{(q)} |^2_{\olg} \right) \right) |_{x \, e^{\phi_\e^{(q)}}= \e },
\] 
where $\phi_{0, \e}^{(0)} \equiv 0$ and $\phi_\e^{(q)}$ is the Hamilton-Jacobi extension of $\phi_{0,\e}^{(q)}$.

The right hand side is a second order (nonlocal) nonlinear operator which depends smoothly on $\e$; furthermore, all
functions are smooth, so all calculations may be done formally. Using that the error term for 
$\phi_{0,\e}^{(0)}  \equiv 0$ is 
\[
\tilde \calN (0, \e ) +1 =  \calO (\e^2)
\]
we deduce successively that
\[
\|\phi_{0, \e}^{(q)} \|_{\mathcal C^{2, \alpha}} \leq c_q \, \e, \qquad \mbox{and} \quad 
\tilde \calN (\phi_{0, \e}^{(q)} , \e ) +1 =  \calO (\e^{2+q})
\]
for all $q$. Note that this iteration scheme is inappropriate to actually solve the equation since at 
each step we lose two derivatives. 

Using the new bdf corresponding to the metric $h_{0, \e}^{(q)}  := e^{2\phi_{0, \e}^{(q)}} \, h_0$, the equation we now
must solve is 
\begin{equation}
\tilde \calN^{(q)}_\e  (\phi_0, \e)   = -1,
\label{eq:eqcase31}
\end{equation}
where $\calN^{(q)}_\e $ corresponds to $\tilde \calN$ when $h_0$ is replaced by $h_{0, \e}^{(q)}$. 
We have arranged that $\tilde \calN^{(q)}_\e  (0 , \e)  + 1 = {\mathcal O} (\e^{q+2})$.

\medskip
\noindent{\bf Estimate on the resolvent}  To simplify notation, drop the indices $q$ and $\e$; thus, for example,
we write $h$ instead of $h^{(q)}_\e$, etc.  The linearization of $\tilde \calN$ at $\phi_0=0$ (for $\e \geq 0$) is 
\[
L_\e   : = \e \, \Delta_{\e} +  q_\e ,
\]
where 
\[
\Delta_\e := \Delta_{h(\e)}, \qquad  \mbox{and} 
\qquad  q_\e : =  -  \tr^{h} \del_x h (\e) = 1 + \calO(\e). 
\]  
Define
\[
\calR = \{\e: \   0 \in \mbox{spec\,}(- L_\e ) \};
\]
 thus $L_\e$ fails to be invertible if and only if $\e \in \calR$. There are two closely related issues: 
to show that $\calR$ is countable and accumulates only at $0$, and to estimate the size of the sets 
 \begin{equation}
J(N) = \{\e \notin \calR \quad : \quad \| L_\e^{-1} \|_{L^2 \to L^2} \leq \e^{-N} \}.
\label{eq:JNA}
\end{equation}
Both facts rely on the observation that as $\e \searrow 0$, $L_\e$ is well-approximated by 
$\e \, \Delta_{h_0} +1$, and the eigenvalues of this latter operator cross $0$ with speed $1/\e$. 
We make this more precise now.

\begin{lemma}
The set $\calR$ consists of an infinite decreasing sequence $\{\e_j\}$ accumulating only at $0$ and has 
counting function  $N(\e) = |\{\e_j \geq \e\}|$ which satisfies $C_1 \e^{-n/2} \leq N(\e) \leq C_2 \e^{-n/2}$.
Furthermore, for each fixed $N  > \frac{n-2}{2}$, 
\[
| J(N,A) \cap (0, \e) |\leq \e - C \, \e^{N -\frac{n-2}2}
\]
for some constant $C$ depending on $N$ but not $\e$. 
\end{lemma}

\begin{proof}
If $\lambda(\e)$ is an eigenvalue in $(-1/2,1/2)$ and is simple with corresponding eigenfunction $\psi (\e)$ 
with $L^2$ norm equal to $1$, then
\[
\begin{array}{rllll}
\dot{\lambda}  & =&  \displaystyle - \int_{\partial M} \psi ( \Delta_\e + \e \, \dot{\Delta}_\e + \dot q_\e ) \psi \, dV_{h(\e)}\\[3mm]
& =&  \displaystyle \frac{\lambda + 1}{\e}  + \int_{\partial M}  \left( \frac{q_\e - 1}{\e} +  \dot q_\e\right)  
\psi^2  \, dV_{h(\e)} + \e \, \int_{\partial M} \psi \, \dot{\Delta}_\e \psi \, dV_{h(\e)} .
\end{array}
\]
As $\e \searrow 0$, both $\frac{q_\e - 1}{\e}$ and $\dot q_\e$ are uniformly bounded.  Writing 
$\Delta_\e \psi =  - \frac{\lambda+ q_\e}{\e} \, \psi$, then boundedness of $\lambda$ and $q_\e$ and 
elliptic estimates show that
\[
\e \, \| \psi \|_{H^2} \leq C \, \| \psi\|_{L^2} 
\Longrightarrow \left| \e \, \int_{\partial M} \psi \, \dot{\Delta}_\e \psi \, dV_{h(\e)} \right| \leq C.
\]
All of this implies that
\[
\left| \dot{\lambda} -\frac{\lambda + 1}{\e} \right| \leq C 
\]
with $C$ independent of $\e$.  Even when the eigenspace is not simple, we can interpret $\dot{\lambda}$ as a 
set-valued function, cf. \cite{C}, \cite{K}, which accomodates the possibility that $\lambda$ splits into a 
number of separate eigenvalues. The estimate for the elements of this set of derivatives remains the same. 

We have proved that if $\e$ is small enough and $\lambda(\e) \in (-1/2, 1/2)$ then  $\dot \lambda \sim  
\frac{\lambda + 1}{\e} $, and in particular, $\dot{\lambda} > 0$. This shows that the set of eigenvalue 
crossings, i.e. values $\e$ where $\lambda (\e) =0$, is discrete, but in fact that the number of eigenvalues 
$\lambda_j(\e)$ of $-L_\e$ which are less than $1/2$ is bounded by $C \,  \e^{- \frac{n}{2}}$; this follows 
directly from the Weyl asymptotic law for the convergent family of metrics $h(\e)$. 

The same estimates give good control on the sets $J(N)$. Indeed, define $I$ to be the set of $j$ such that the 
length of $(\e_{j+1}, \e_j)$ is larger than $4\, \e^{N+1}$ and this interval intersects $(\e, 2\e)$. The estimate above
for $\dot{\lambda}$ implies that if $j \in I$ and $\tilde \e \in (\e_{j+1} + \e^{N}, \e_j - \e^{N})$, then all 
eigenvalues of $-L_{\tilde \e}$ are at least at distance  $\tilde \e^{-N} $ from $0$, and hence $\tilde \e \in J (N)$ 
since $\|(-L_{\tilde \e})^{-1}\|_{L^2} \leq \tilde \e^{-N}$ by the spectral theorem. 

The number of intervals $(\e_{j+1}, \e_j)$ which intersect $(\e, 2\e)$ is bounded by $C \, \e^{-\frac{n}{2}}$, so
the complement in $(\e, 2\e)$ of the union of intervals $(\e_{j+1} + \e^{N}, \e_j - \e^{N}) $ with $j \in I$ 
covers at most $C \, \e^{N+1-\frac{n}{2}}$ of the length of $(\e, 2\e)$. This completes the proof. 
\end{proof}

To convert this from an $L^2$ estimate to one between $\e$-scaled H\"older spaces, revert to the scaled 
metric $\tilde{h}(\e)$ and note that it has volume proportional to $\e^{-n}$. Local elliptic estimates, which 
are uniform for balls $B$ of size $1$ in $(\del M, \tilde{h}(\e))$, give that
\[
\| u \||_{{\mathcal C}^{2,\alpha}} \leq C \left( \|f \|_{{\mathcal C}^{0,\alpha}} + \|u \|_{L^2} \right).
\]
However, $||u||_{L^2} \leq C \e^{-N} \, \|f \|_{L^2} \leq C \e^{-N-n} \, \|f \|_{{\mathcal C}^{0,\alpha}}$, so this proves the
\begin{lemma}
If $\e \in J(N,A)$, then the norm of $(-L_\e)^{-1}$ as a map between $\calC^{0,\alpha}_\e$ 
and $\calC^{2,\alpha}_\e$ is bounded by $C \e^{-N-n}$ for some constant $C$ which is independent of $\e$.
\end{lemma}

\medskip

The rest of the proof now proceeds as follows. First fix $N > \frac{n+2}{2}$ and $q > N+n+1$, and use the 
approximate solution $h^{(q)}_{0,\e}$. This will be perturbed using a fixed point argument. The key fact is that 
the norm of the inverse of the linearization of  (\ref{eq:eqcase31}) is now bounded by $C \, \e^{-N-n-1}$ for 
some fixed $C$ and for all $\e \in J(N)$. The same proof works to find a solution $\phi_0$ of 
(\ref{eq:eqcase31}) lying in a ball of radius $C \e^{q+2-N-n}$.

\subsubsection*{Foliations}
We conclude this section by proving that the CMC hypersurfaces constructed in each of these
three cases are the leaves of a foliation.

Let $\Sigma$ be any one of the CMC hypersurface constructed above. Any other hypersurface $\Sigma'$ which 
is nearby to $\Sigma$ in the $\calC^1$ norm can be written as a normal graph over it, i.e. 
\[
\Sigma' = \{p + \psi(p)N(p): p \in \Sigma\}.
\]
Slightly more generally, a smooth family $\{\Sigma_\eta\}$ of nearby hypersurfaces correspond to a family of 
functions $\psi_\eta$ for which they are the normal graphs. Let us write the mean curvature functions $H(\eta)$
of these hypersurfaces as some nonlinear elliptic operator $M(\psi_\eta)$. Suppose now that we have some 
information about how these (possibly nonconstant) mean curvatures vary with $\eta$. Differentiating this
equation with respect to $\eta$ gives the formula
\begin{equation}
\calL_{\Sigma} \dot{\psi} = \del_\eta H(\eta);
\label{eq:jov}
\end{equation}
here 
\[
\calL_{\Sigma} = \e \,  \Delta_{h(\e)} +  ||\II||^2 + \Ric(N,N)
\]
is the Jacobi operator for the mean curvature function, $\dot{\psi}$ is the derivative of $\psi_\eta$
with respect to $\eta$ at $\eta = 0$, and the right hand side is the derivative of the mean curvature
function with respect to $\eta$. 

Let $x$ be the special bdf associated to $\Sigma$ normalized so that $\Sigma = \{x=\e\}$, say.
We first apply (\ref{eq:jov}) when $\Sigma_\eta = \{x = \e + \eta\}$; in this case, $\dot{\psi} \equiv 1$,
so we obtain that
\[
\calL_\Sigma 1 = ||\II||^2 + \Ric(N,N) = \del_\e H(\e),
\]
where $H$ is the mean curvature function for the level sets $\{x = \mbox{const.}\}$. However, this is
given explicitly in (\ref{eq:H2}), so we deduce that
\begin{equation}
||\II||^2 + \Ric(N,N) = - \kappa_1  - 2 \kappa_2  \, \e + \calO (\e^2).
\label{eq:defqe}
\end{equation}
Let us denote this potential for the Jacobi operator by $q_\e$. 

The simplest case to understand is when $q_\e < 0$ everywhere, which by (\ref{eq:defqe}) is equivalent to 
assuming that either $\kappa_1 > 0$, or else $\kappa_1 \equiv 0$ and $\kappa_2 >0$.
Now, at the risk of repeating notation, let $\{\Sigma_\eta\}$ denote the family of CMC hypersurfaces near to 
$\Sigma$, and $\psi_\eta$ the corresponding Normal graph functions. Applying (\ref{eq:jov}) again
shows that
\[
\calL_\Sigma \dot{\psi} = \del_\eta H(\Sigma_{\eta});
\]
when $\kappa_1 > 0$, the right hand side is simply $-1$, while in the other situation, it equals $-2\e$,
but in either case is strictly negative. Because the potential term in $\calL_\Sigma$ is negative,
the maximum principle implies that $\dot{\psi} > 0$, and this is obviously equivalent to the 
fact that the hypersurfaces $\Sigma_\eta$ are one-sided perturbations, and hence this family
forms a foliation.

The remaining cases are when $\kappa_1 < 0$, or else $\kappa_1 \equiv 0$ and $ \kappa_2 < 0$. 
The maximum principle no longer applies, so we must proceed slightly
differently. The idea now is to show that if $\Sigma_\e$ is the CMC hypersurface which is obtained
as a perturbation of the level set $\{x = \e\}$, then the function $\psi_\e$ which represents
$\Sigma_\e$ as a normal graph over that level set is of size $\e^2$, along with all of its 
derivatives (with respect to the coordinates $y$ on $\del M$).  Of course, we did not construct
$\Sigma_\e$ via this graph function, but rather as the level set $x e^{\phi_\e} = \e$, where
$\phi_\e$ is the solution of the appropriate nonlinear equation we obtained. The translation
between the two representations is not so difficult, and in fact we see that the estimate
$\psi_\e = \calO(\e^2)$ (along with all its derivatives) follows directly from the fact that
$\phi_\e = \calO(\e)$ (again along with all derivatives), which in turn is a direct consequence
of the ball in which the contraction argument was applied in order to find the solution. 
From these estimates, it is now straightforward that these CMC hypersurfaces form a foliation
in these other cases too. 

\section{Other curvature functions}
In this brief final section we sketch some of the ideas needed to extend the methods and results
of this paper to construct other Weingarten foliations, and in particular, foliations where the
leaves have constant $\sigma_k$ curvature. For simplicity we focus only on these latter
functionals. 

The preliminary work is identical. As before, we replace the boundary metric $h_0$ by $\wh{h}_0 = e^{2\phi_0}h_0$, 
let $\wh{x}$ denote the corresponding special bdf, and calculate the second fundamental form $\II(\phi_0)$
of the level sets $\{\wh{x} = \mbox{const.}\}$ as in (\ref{eq:sff1}). However, instead of taking
the trace, now apply the $\sigma_k$ functional, i.e.\ take the $k^{\mathrm{th}}$ symmetric
function of the eigenvalues of $\II$ with respect to the induced metric on each level set. 
This is the more complicated fully nonlinear operator
\[
\calN_k(\phi_0,\e) = \left. e^{-2k\phi}\sigma_k^{\olg}\left( (h - \frac12 x\del_x h)  - x dx \circ d\phi -
x^2\left( \mbox{Hess}^{\olg}\phi + \frac12|\nabla^{\olg}\phi|^2_{\olg}\right)\right)\right|_{xe^{\phi} = \e}.
\]

In order to calculate the asymptotics of this functional when $\phi = 0$ and $\e \searrow 0$, and its derivative 
with respect to $\phi_0$ at $\phi_0 = 0$, we use the following formul\ae: if $B(s)$ is any one-parameter family
of symmetric matrices, then
\begin{equation}
\frac{d\,}{ds}\sigma_k( B(s)) = \tr \left(\dot{B}(s) T_{k-1}(B(s))\right),
\label{eq:dsk1}
\end{equation}
where 
\[
T_{k-1}(B) = \sum_{j=0}^{k-1} (-1)^j \sigma_{k-1-j}(B) B^j
\]
is the Newton polynomial of order $(k-1)$ of $B$. Differentiating again gives
\begin{multline}
\frac{d^2\, }{ds^2} \sigma_k(B(s)) =  \\ \tr \left(\ddot{B} \, T_{k-1}(B) + \dot{B}
\sum_{j=0}^{k-1} (-1)^j \left(\tr(\dot{B} \, T_{k-2-j}(B)) + \sigma_{k-1-j}(B)\, j\, B^{j-1}\, \dot{B}\right)\right).
\label{eq:dsk2}
\end{multline}
In the present setting, if $A$ is a symmetric $2$-tensor, $\sigma_k^g(A)$ represents the $\sigma_k$ functional on the
$(1,1)$ tensor $B$ obtained by raising one index of $A$ using the metric $g$. The formul\ae\ (\ref{eq:dsk1})
and (\ref{eq:dsk2}) are interpreted accordingly. 

We apply this in two different ways. First, we calculate the expansion of 
\[
S_k(x) := \sigma_k^{g}(\II(0)) = \sigma_k^{\olg}(x^2\II(0)) = 
\sigma_k^{h(x)}(h_0 + \frac12 h_1 x + \calO(x^3)) = \sigma_k(B(x))
\] 
where \[
B(x)_i^{\ j} = \delta_i^{\ j} - \frac12 (h_1)_i^{\ j}x - \left( (h_2)_i^{\ j} -\frac12 (h_1 \circ h_1)_i^{\ j}\right) x^2 
+ \calO(x^3). 
\]
After some work, we find that 
\[
S_k(x) = \binom{n}{k} -\binom{n-1}{k-1}\kappa_1 x  + \left( -2 \binom{n-1}{k-1}\kappa_2 + 
\frac12 \binom{n-2}{k-2} \sigma_2^{h_0}(h_1) \right) x^2 + \calO(x^3), 
\]
where $\kappa_1$ and $\kappa_2$ are precisely the same functions as we have been considering before. 
Anyone attempting to verify this should take advantage of the two combinatorial formul\ae\ 
\[
T_\ell(I) = \binom{n-1}{\ell} \, I, \qquad \sum_{j=0}^\ell (-1)^j j \binom{n}{\ell-j} = - \binom{n-2}{\ell-1}. 
\]
Similarly,
\begin{multline}
\left. D_1 \sigma_k^{\olg}\left(x^2 \II(0) - x\, dx \circ d\phi - x^2 (\mbox{Hess}^{\olg}\phi + 
\frac12 |\nabla^{\olg}\phi|_{\olg}^2\olg) \right)\right|_{0}(\psi_0) \\ 
= \tr^{h(x)} \left( (-x\, dx \circ d\psi_0) - x^2 \mbox{Hess}^{h(x)}\psi_0)T_{k-1}( x^2 \II(0)) \right).
\notag 
\end{multline}
Note that the first term in this last expression is always off-diagonal, hence does not contribute. 

In the interest of space and with a mind to the law of diminishing returns, we focus on the 
weakly Poincar\'e-Einstein case. Since $h_1 = 0$, we have
\[
\sigma_k^{\olg}(x^2 \II(0)) = \binom{n}{k} + \frac{1}{n-1} \binom{n-1}{k-1} R_{h_0} x^2 + \calO(x^3),
\]
and we assume that $R_{h_0}$ is a (nonzero) constant too. The equation we must solve, then, is 
\[
\calN_k(\phi_0,\e) = \binom{n}{k} + \frac{1}{n-1} \binom{n-1}{k-1} R_{h_0} \e^2
\]
or equivalently,
\begin{equation}
\frac{1}{\e^2}\left(\calN_k(\phi_0,\e) - \binom{n}{k}\right) = \frac{1}{n-1} \binom{n-1}{k-1} R_{h_0}. 
\label{eq:skeqts}
\end{equation}
Using the formula above for the linearization of $\sigma_k$, we see that the principal part as $\e \searrow 0$ 
of the linearization of the operator on the left in this final expression at $\phi_0 = 0$ is 
\[
\binom{n-1}{k-1}\left(\Delta_{h_0} + \frac{R_{h_0}}{n-1}\right). 
\]
As expected, this is invertible when $R_{h_0} < 0$.  Assuming invertibility of this operator,
we are able to apply the implicit function theorem and find a solution exactly as before. 

We summarize this discussion in the
\begin{theorem}
Let $(M,g)$ be conformally compact and weakly Poincar\'e-Einstein. If the conformal infinity
$\frakc(g)$ has negative Yamabe invariant, then for each $k = 1, \ldots, n$, there is a 
unique foliation near infinity in $M$ by hypersurfaces with constant $\sigma_k$ curvature.
If $\frakc(g)$ is positive, then for each $h_0$ in this conformal class with constant
(positive) scalar curvature for which the conformal Laplacian is nondegenerate, and for
each $k  = 1, \ldots, n$, there is an associated foliation.
\end{theorem}
As mentioned in the introduction, the special case of greatest interest is the
\begin{corollary}
Let $M$ be a quasi-Fuchsian $3$-manifold. Then each end of $M$ admits a unique foliation
by constant Gauss curvature surfaces. 
\end{corollary}

\end{document}